\documentclass[12pt]{amsart}

\usepackage{amsmath,amsthm,amssymb,amsfonts}
\usepackage{mathrsfs}
\usepackage{mathtools}
\usepackage{bm} \usepackage{upgreek} \usepackage{pifont} \usepackage{blkarray} \usepackage[all]{xy} \usepackage{enumitem} \usepackage{subdepth} \usepackage{fouridx} 
\usepackage{graphicx}
\usepackage{tikz}
\usepackage{calc}
\usetikzlibrary{calc}
\usetikzlibrary{positioning}
\usetikzlibrary{decorations.pathmorphing,decorations.markings}
\usetikzlibrary{arrows.meta}
\usetikzlibrary{backgrounds}
\usetikzlibrary{fit}

\usepackage[margin = 1.2in]{geometry}

\usepackage[colorlinks,linkcolor=red,anchorcolor=blue,citecolor=green]{hyperref}
\usepackage[capitalise,noabbrev]{cleveref}

\newtheorem{theorem}{Theorem}

\newtheorem{definition}[theorem]{Definition}

\newtheorem{lemma}[theorem]{Lemma}

\newtheorem{remark}[theorem]{Remark}

\newtheorem{proposition}[theorem]{Proposition}

\DeclarePairedDelimiter{\set}{\{}{\}}
\DeclarePairedDelimiter{\floor}{\lfloor}{\rfloor}

\begin{document}

\title[Half of an antipodal spherical design]{Half of an antipodal spherical design}

\author{Eiichi Bannai}

\address{School of Mathematical Sciences, Shanghai Jiao Tong University, Shanghai, China.}

\email{bannai@sjtu.edu.cn}

\author{Da Zhao}

\address{School of Mathematical Sciences, Shanghai Jiao Tong University, Shanghai, China.}

\email{jasonzd@sjtu.edu.cn}

\author{Lin Zhu}

\address{School of Mathematical Sciences, Shanghai Jiao Tong University, Shanghai, China.}

\email{zhulin2323@163.com}

\author{Yan Zhu}

\address{School of Mathematical Sciences, Shanghai Jiao Tong University, Shanghai, China.}

\email{zhuyan@sjtu.edu.cn}

\author{Yinfeng Zhu}

\address{School of Mathematical Sciences, Shanghai Jiao Tong University, Shanghai, China.}

\email{fengzi@sjtu.edu.cn}



\begin{abstract}
    We investigate several antipodal spherical designs on whether we can choose half of the points, one from each antipodal pair, such that they are balanced at the origin. In particular, root systems of type A, D and E, minimal points of Leech lattice and the unique tight 7-design on $S^{22}$ are studied. We also study a half of an antipodal spherical design from the viewpoint of association schemes and spherical designs of harmonic index $T$. 
\end{abstract}

\maketitle

\section{Introduction} \label{sec:Breakfast}
    We first study the following two problems:
    \begin{enumerate}[label=$(\arabic*)$]
        \item \label{itm:Indication} Consider the root system of type $E_8.$ 
        Can we choose a half of the 240 vectors by picking one from each antipodal pair so that their sum is the zero vector?
        \item \label{itm:Kilt} Consider 196560 minimum vectors of Leech lattice. 
        Can we choose a half of the 196560 vectors by picking one from each antipodal pair so that their sum is the zero vector?
    \end{enumerate}
    
    Since these problems are very simply minded, we thought they might have been studied by some people already. But we could not find them in the literature. 
    Also, the solutions were not so obvious nor easy to get in spite of the innocent looking nature of the problem. So we decided to present the results together with the explanation of our motivation why we were interested in these problems.

    From now on it is assumed that we take one from each antipodal pair when we say `half of', unless stated otherwise.

    In Section~\ref{sec:Control}, we give a positive answer to Problem~\ref{itm:Indication}. In fact, we discuss more generally when this property holds for the root vectors of type $A_l, D_n$ and $E_m$. 
    In Section~\ref{sec:Halibut}, we give a positive answer to Problem~\ref{itm:Kilt}. 
    Then, in Section~\ref{sec:Travel}, we explain why we have been interested in these problems.
    In fact, the property that the sum of these unit vectors is the origin $\bm{0}$ is nothing but they form a spherical 1-design. Our interest has arisen when we studied spherical designs of harmonic index $T$ for some subset $T$ of natural numbers. 
                                We also discuss these half subsets considered in Problem~\ref{itm:Indication} and Problem~\ref{itm:Kilt}, from the viewpoint of association schemes.

\section{\texorpdfstring{Case $A_l$ ($l\geq 1$), $D_n$ ($n\geq 4$), $E_m$ ($m=6, 7, 8$)}{Case Al, Dn, Em} } \label{sec:Control}
    In this section we consider a slightly general problem whether we can choose half roots from a root system 
        so that their sum is zero. Note that it is equivalent to ask that whether one can assign plus signs and minus signs in front of all positive roots, such that their sum equals $\bm{0}$. It is also equivalent to asking the question that whether one can choose some positive roots such that their sum equals half sum of the positive roots, which is usually denoted by $\rho$ in Lie theory.

    In this section we will determine which root system of type A, D or E has such property. 

    \begin{proposition} \label{pro:halfOfAl}
        There exists a half of $A_l$ such that their sum is $\bm{0}$ if and only if $l \equiv 0 \pmod{2}$.
    \end{proposition}

    \begin{proposition} \label{pro:halfOfDn}
        There exists a half of $D_n$ such that their sum is $\bm{0}$ if and only if $n \equiv 0,1 \pmod{4}$.
    \end{proposition}

    \begin{proposition} \label{pro:halfOfEm}
        There exists a half of $E_6$ or $E_8$ such that their sum is $\bm{0}$ but not for $E_7$.
    \end{proposition}
    In the following proofs, we denote by $\{\varepsilon_i \}_{i=1}^n$ the standard basis of $\mathbb R^n$. 
    \begin{proof}[Proof of \cref{pro:halfOfAl}]
    For $A_l$ case, a simple root system $\Delta$ is \{$\varepsilon_1-\varepsilon_2, \varepsilon_2-\varepsilon_3, \cdots, \varepsilon_l-\varepsilon_{l+1}$\}, and the corresponding positive root system $\Pi$ is \{$\varepsilon_i-\varepsilon_j \mid 1\leq i< j\leq l+1$\}. So $$\rho=\frac{1}{2}\sum\limits_{\alpha \in \Pi} \alpha=\frac{1}{2}(l\varepsilon_1+(l-2)\varepsilon_2+\cdots+(-l+2)\varepsilon_l+(-l)\varepsilon_{l+1}).$$
    Consider the coefficient of $\varepsilon_1$, one can easily see that there does not exist a set of certain positive roots such that their sum equals $\rho$ if $l$ is odd.

    If $l=2k$, we have
    \[
        \sum_{1\leq i < j \leq 2k+1} (-1)^{i+j}(\varepsilon_i - \varepsilon_j)=\bm{0}.
    \]
                                                    Consequently one can choose half roots from $A_l$ 
        such that their sum is zero if and only if $l$ is even.
    \end{proof}

    \begin{proof}[Proof of \cref{pro:halfOfDn}]
    For $D_n$ case, a positive root system $\Pi$ is \{$\varepsilon_i\pm\varepsilon_j \mid 1\leq i< j\leq n$\}. So $$\rho=\frac{1}{2}\sum\limits_{\alpha \in \Pi} \alpha=(n-1)\varepsilon_1+(n-2)\varepsilon_2+(n-3)\varepsilon_3+\cdots+2\varepsilon_{n-2}+\varepsilon_{n-1}.$$
    Note that the sum of the coefficients of these $\varepsilon_i$'s in $\rho$ is $\frac{n(n-1)}{2}$. If $\rho$ can be written as the sum of some positive roots, then the sum of the coefficients of these $\varepsilon_i$'s in $\rho$ must be an even number.
        So $\rho$ is in the root lattice $D_n$ if and only if $\frac{n(n-1)}{2}$ is even, showing the `only if' direction of \cref{pro:halfOfDn}. For the other direction we have the following two identities.

    If $n=4k$, we have the identity
    \[
        \sum_{1\leq i < j \leq 4k} \left[ (-1)^{\floor{\frac{i+j}{2}}}(\varepsilon_i + \varepsilon_j)
         + (-1)^{i+j}(\varepsilon_i - \varepsilon_j)\right] = \bm{0}.
    \]
                                                                                        
    If $n=4k+1$, we have the identity
                \begin{align*}
        \bm{0} & = \sum_{i=1}^{4k} (-1)^i \Big[ (\varepsilon_i + \varepsilon_{4k+1}) - (\varepsilon_i - \varepsilon_{4k+1}) \Big] \\
        & + \sum_{1\leq i < j \leq 4k} \left[ (-1)^{\floor{\frac{i+j}{2}}}(\varepsilon_i + \varepsilon_j) + (-1)^{i+j}(\varepsilon_i - \varepsilon_j)\right].
    \end{align*}
                                                                                    
    Consequently, one can choose half roots from $D_n$ 
        such that their sum is zero if and only if $n=4k$ or $n=4k+1$. 
        \end{proof}

    \begin{proof}[Proof of \cref{pro:halfOfEm}]
        For $E_8$ case, by \cite[pp.~120-121]{CS99} or \cite[p. 283]{Bou02} we have $E_8 = E_8^a \cup E_8^b$, where $E_8^a = \set{\pm \varepsilon_i \pm \varepsilon_j \mid 1 \leq i < j \leq 8}$ and $E_8^b$ consists of some roots with even number of negative coefficients 
        \[
            E_8^b = \set*{\frac{1}{2}\sum_{i=1}^8 c_i \varepsilon_i \;\middle\vert\; c_i\in \{\pm 1\},\prod_{i=1}^8c_i=1}.
        \]
                                                        Note that $E_8^a$ is exactly $D_8$, which has been discussed in \cref{pro:halfOfDn}.
        So we only need to prove that $E_8^b$ has the desired property. 
        Say we choose all those roots such that $(c_1,c_2,c_3)$ are among $(1,1,1)$, $(1,-1,-1)$, $(-1,1,-1)$ and $(-1,-1,1)$. One can check that we take one from each antipodal pair and their sum is indeed ${\bm 0}$.

        For $E_6$ case, by \cite[pp.~125-127]{CS99} we have $E_6 = E_6^a \cup E_6^b$, where $E_6^a = D_5$ and $E_6^b$ consists of some roots with odd number of negative coefficients
        \[
            E_6^b = \set*{\frac{\sqrt{3}}{2} c_6\varepsilon_6 + \frac{1}{2}\sum_{i=1}^5 c_i \varepsilon_i  \;\middle\vert\; c_i\in \{\pm 1\},\prod_{i=1}^6c_i=-1}.
        \]
                                                        The first part is isomorphic to $D_5$ and a similar argument as in the case $E_8$ shows that the second part has the desired property.
    
        For $E_7$ case, it can be represented in a $7$-dimensional subspace of $\mathbb{R}^8$ \cite[pp.~124-125]{CS99}, i.e. 
        \[
            E_7 = \left\{ \varepsilon_i - \varepsilon_j \;\middle\vert\; 1 \leq i \neq j \leq 8 \middle\} \bigcup \middle\{ \frac{1}{2} \sum_{i=1}^8 c_i \varepsilon_i  \;\middle\vert\; c_i\in \{\pm 1\},\sum_{i=1}^8c_i=0\right\}.
        \]
                In order to get a total sum of $\bm{0}$, the total coefficient of $\varepsilon_1$ should be $0$ in particular. But we cannot get an integer from the sum of thirty-five $\frac{1}{2}$ or $-\frac{1}{2}$'s.
                            \end{proof}

    Finally we show that half of $E_8$,
        say $X$, cannot have the structure of an association scheme induced by inner product.
        Suppose $\mathscr{X}=(X,\{R_i\}_{0\leq i \leq 3})$ is indeed one such association scheme where relations are given by $R_i=\{(x,y) \mid \langle x, y\rangle=2-i\},i=0,1,2,3$. Let $p_{ij}^k$ be the corresponding intersection numbers in $X$ and $\widetilde{p}_{ij}^k$ be the corresponding intersection numbers in $E_8$ as an association scheme.  
    Since we choose one point from each antipodal pair, for every fixed $x,y\in X$, we obtain the following relation.
        \begin{align*}
            &\{z\mid\langle z,x\rangle=1,\langle z,y\rangle=-1,z\in E_8\}\\
            =&\{z\mid\langle z,x\rangle=1,\langle z,y\rangle=-1,z\in X\}
            \bigcup\{z\mid\langle z,x\rangle=1,\langle z,y\rangle=-1,-z\in X\}\\
            =&\{z\mid\langle z,x\rangle=1,\langle z,y\rangle=-1,z\in X\}
            \bigcup\{z\mid\langle z,x\rangle=-1,\langle z,y\rangle=1,z\in X\}.
        \end{align*}
    This implies $p_{13}^k+p_{31}^k=\widetilde{p}_{13}^k$, where $k=0,1,2,3$. Since $X$ is a symmetric association scheme, we have $p_{13}^k=p_{31}^k=\frac{1}{2}\widetilde{p}_{13}^k$, where $k=0,1,2,3$. In $E_8$ we have $\widetilde{p}_{13}^1=1$, hence $p_{13}^1=\frac{1}{2}$ which is a contradiction. Therefore we get the following proposition.

    \begin{proposition}
        Every half of $E_8$ (normalized to unit vectors) cannot have the structure of an association scheme of class 3 with inner product $0,\frac{1}{2},-\frac{1}{2}$, even without the assumption that their sum is zero.
    \end{proposition}

\section{Half of minimum vectors of Leech lattice} \label{sec:Halibut}

    There are many approaches to construct the minimum vectors of the Leech lattice, denoted by $\Lambda_{24}(1)$, here we use the following one which relates to the extended Golay code $\mathscr{C}$: The 196560 minimal vectors of Leech lattice could be divided into the following three types \cite[p.~286]{CS99}:

    \begin{enumerate}[label=(\alph*)]
        \item $\Lambda_{24}^{a}(1)$: 1104 vectors of type $(4^2,0^{22})$, with all permutations and all $\pm$ sign choices; \label{itm:l24a}
        \item $\Lambda_{24}^{b}(1)$: 97152 vectors of type $(2^8,0^{16})$, where the positions of `2's correspond to the 759 octads in the extended Golay code, and the choices of $\pm$ signs are restricted by the condition that the total number of minus signs are even; \label{itm:l24b}
        \item $\Lambda_{24}^{c}(1)$: 98304 vectors of type $(-3,1^{23})$, where the `3' can appear in each position and the positions where the $\pm$ signs change correspond to the `1's in the extended Golay code. \label{itm:l24c}
    \end{enumerate}

    Please recall that the extended Golay code $\mathscr{C}$ in $\mathbb{F}_2^{24}$ is generated by the matrix $G=[I\ P]$, where $I$ is the identity matrix of order $12$ and $P$ equals the all `1' matrix $J$ minus the adjacency matrix of icosahedron graph.

    Now we deal with these three types separately. For type \ref{itm:l24a}, it reduces to the case $D_{24}$; For type \ref{itm:l24b}, it reduces to the 759 collections of the set $S$ of case $E_8$; For type \ref{itm:l24c}, we need the following lemma.

    \begin{lemma}
        There exists a half of $\mathscr{C}$, one from each complementary pair, such that their sum equals half of the total sum of $\mathscr{C}$.
    \end{lemma}

    \begin{proof}
        We denote by $\mathscr{C}_{ijk}$ the codewords in $\mathscr{C}$ which start with $ijk$, where $ijk$ is a 0-1 string of length 3. We take the half to be $\mathscr{C}_{111}\cup\mathscr{C}_{100}\cup\mathscr{C}_{010}\cup\mathscr{C}_{001}$. Their complement codewords are exactly $\mathscr{C}_{000}\cup\mathscr{C}_{011}\cup\mathscr{C}_{101}\cup\mathscr{C}_{110}$. We denote by $\bm{x}_{ijk}$ the row vector ${}^t(i,j,k,0,\dots,0)$ in $\mathbb R^{12}$ and by $\bm{c}_{ijk}=\bm{x}_{ijk}G$ the corresponding codeword. They have the desired property due to the identity $\bm{c}_{111}+\bm{c}_{100}+\bm{c}_{010}+\bm{c}_{001}=\bm{c}_{000}+\bm{c}_{011}+\bm{c}_{101}+\bm{c}_{110}.$
    \end{proof}

    Since the `1's in the extended Golay code correspond to the $\pm$ sign changes in $\Lambda_{24}^{c}(1)$, thus for each one of the 24 collections of vectors in $\Lambda_{24}^c(1)$ with fixed position of $\pm 3$, we can choose half of them, one from each antipodal pair, such that their sum is $\mathbf{0}$. Combine these three types, we get the following proposition.

    \begin{proposition}
        There exists a half of $\Lambda_{24}(1)$ such that their sum is $\mathbf{0}$.
    \end{proposition}

    Now we show that half of $\Lambda_{24}(1)$ cannot be an association scheme induced by inner product. Suppose it is, say the relations $R_0,R_1,R_2,R
    _3,R_4,R_5$ are given by inner products $1,0,\frac{1}{4},-\frac{1}{4},\frac{1}{2},-\frac{1}{2}$ respectively. By the same argument in Section~\ref{sec:Control}, we have $p_{45}^4=\frac{1}{2}$, which is impossible. So we get the following proposition.

    \begin{proposition}
        Every half of minimal vectors of Leech lattice (normalized to unit vectors) cannot have the structure of an association scheme of class 5 with inner product $0,\frac{1}{4},-\frac{1}{4},\frac{1}{2},-\frac{1}{2}$, even without the assumption that their sum is zero.
    \end{proposition}
    
\section{Spherical designs of harmonic index \texorpdfstring{$T$}{T}} \label{sec:Travel}
    
    In this section, we shall explain our motivation of these problems and the relationship with the spherical designs of harmonic index $T$.    
            \begin{definition}
     Let $T$ be a subset of $\mathbb N=\{1,2,\ldots\}$.
     A nonempty subset $Y$ of the unit sphere $S^{n-1}$ is called a spherical $T$-design (or spherical design of harmonic index $T$) if the following condition is satisfied
    $$\sum_{\bm{x}\in X}f(\bm{x})=0$$
    for every homogeneous harmonic polynomial $f(\bm{x})=f(x_1,x_2,\ldots ,x_n)$ of degree $i$ with $i\in T.$  
    \end{definition}
    In particular, if $T=\{1, 2,\ldots ,t\},$ then $Y$ is a usual spherical $t$-design which was introduced and well studied by Delsarte-Goethals-Seidel \cite{DGS77}.
    The systematic study on spherical designs of harmonic index $T$ was initiated by Bannai-Okuda-Tagami \cite{BOT15} for $T=\{t\}$ and continued by \cite{OY16} and \cite{ZBBKY17}. 
    Delsarte-Goethals-Seidel introduced a linear programming method to determine the Fisher type lower bound for spherical $t$-designs and the method was generalized to $T=\{t_1,t_2,\cdots,t_\ell\}$ with all $t_i$'s being even positive integers (Cf \cite{BOT15,ZBBKY17}).
    
    A spherical $t$-design is called tight if the lower bound is attained.
    A tight $t$-design in $S^1\subset \mathbb R^2$ is a regular $(t+1)$-gon.
    For $n\geq 3$, Bannai-Damerell \cite{BD79,BD80} proved that if there exists a tight spherical $t$-design on $S^{n-1}$, then $t$ must be in $\{1,2,3,4,5,7,11\}$. 
    The classification of tight spherical $t$-designs was complete for $t=1,2,3,11$.
    Moreover, Bannai-Sloane \cite{BS81} obtained the uniqueness of tight spherical 11-design up to orthogonal transformations, namely the 196560 minimal vectors of the Leech lattice on $S^{23}$.
    When $t=4,5,7$, the classification is still an open problem, but there are some known examples (Cf \cite{DGS77}) and some nonexistence results (Cf \cite{BMV05, NV13}).
    For instance, 240 points of the $E_8$ root system on $S^7$ is a tight spherical 7-design.
    We should remark that tight spherical $(2e+1)$-designs are always antipodal which was proved in \cite{DGS77}.

    In \cite{BGOY15}, they use the terminology {\it tight frames} for spherical designs of harmonic index $\{2\}$ and it is shown that if $X$ is a two-distance set tight frame with two inner products $\{a, b\}$, then one of the following two possibilities holds:
     \begin{enumerate}[label=$(\arabic*)$]
    \item \label{itm:Glut} if $a\neq -b$, then it must be either an $n$-dimensional spherical embedding of a strongly regular graph (SRG) or a shifted $(n-1)$-dimensional spherical embedding of a SRG.
    (Please refer to \cite[Example 9.1]{DGS77} or \cite{BB05} for the construction of spherical embedding of SRGs.)
    \item \label{itm:Guest} if $a=-b$, then $X$ can be identified as a half of an antipodal spherical 3-design.
    \end{enumerate}

    Two-distance tight frames with $a=-b$ are said to be equiangular.
    It is interesting to notice that many (but not all, say icosahedron) equiangular tight frames  have the structure of SRGs, in particular we have no explicit example of non-conference type (i.e. $|X|\neq 2n$) tight frame which does not have this property, so far. 
    In \cite{BBXYZ17}, they study a family of equiangular tight frames having this property which are both two-distance sets and spherical designs of harmonic index $\{4,2\}$, namely, the half of tight spherical $5$-designs. 
     On the other hand, such objects are the spherical embedding of some SRGs satisfying the condition for spherical designs of harmonic index $\{4\}$.
        In following homepage of Brouwer, the list of known SRGs is available.
    \begin{center}
    \verb@http://www.win.tue.nl/~aeb/@
    \end{center} 

    The property is equivalent to the condition that $X$ becomes a spherical $\{4, 2, 1\}$-design. 
    We thought it would be natural to consider what kind of other tight spherical $t$-designs (for odd $t$) have this property, and this led to the Problems~\ref{itm:Indication} and \ref{itm:Kilt} mentioned at the beginning of this paper. 

    We would like to summarize some results in \cref{tab:Degree}.

    \begin{table}[h] 
        \begin{tabular}{|c|c|c|c|}
            \hline
                & Existence & Nonexistence & Unknown \\
            \hline
            $E_8$ & 1         & 3,5          & -     \\
            \hline
            tight $7$-design on $S^{22}$ & 1 & 3 &  5  \\
            \hline
            tight $11$-design on $S^{23}$ & 1 & - & 3,5,7,9 \\
            \hline 
        \end{tabular}
        \caption{Half of an antipodal design being a spherical $T$-design}
        \label{tab:Degree}
    \end{table}

    This table shows whether there exists a half of one antipodal spherical $t$-design being an $\set{2,4, \ldots, 2\floor{t/2}} \cup \set{i}$-design, where $i$ is an odd number less than $t$.

    Another type of results is that every half of $X$ do not carry the structure of an association scheme induced by the inner product, where $X$ is one of the following three designs: $E_8$, the tight $7$-design on $S^{22}$, and the tight $11$-design on $S^{23}$ (normalized minimum vectors of Leech lattice).

    Similar questions could be asked for many tight spherical $(2e+1)$-designs on $S^{n-1} \subset \mathbb{R}^n$.

    \begin{remark}
        To show that every half of a tight spherical $(2e+1)$-design cannot be a spherical design of harmonic index $\{i\}$, we consider `half of' the $i$-th \emph{characteristic matrix} $\widetilde{H}_i$.
        For a fixed orthonormal basis $\{\varphi_{i,\alpha}\}_{\alpha\in\Lambda}$ of homogeneous harmonic polynomials of degree $i$ and a fixed half of the design, say $X$, we define $\widetilde{H}_i=\left(\varphi_{i,\alpha}(x)\right)_{x\in X,\alpha\in\Lambda}$. We only need to show that there is no $\pm 1$ vector in the left null space of $\widetilde{H}_i$. It can be done by computer when the scale is not large. One can prove that ${}^t\widetilde{H}_i\widetilde{H}_j=0$ if $i,j$ are distinct odd numbers and $i+j<2e+1$. We observe that $\widetilde{H}_1$, $\widetilde{H}_3$ of half of $E_8$, minimal vectors of Leech lattice and tight $7$-design on $S^{22}$ are all of full column rank.
    \end{remark}

    \section*{Acknowledgments}

    This work was supported in part by NSFC [Grant No. 11271257].

    \bibliography{NT}

\def\SortNoop#1{}
\begin{thebibliography}{10}

\bibitem{BB05}
E.~Bannai and Et. Bannai.
\newblock A note on the spherical embeddings of strongly regular graphs.
\newblock {\em European J. Combin.}, 26(8):1177--1179, 2005.

\bibitem{BBXYZ17}
E.~Bannai, Et. Bannai, Z.~Xiang, W.~Yu, and Y.~Zhu.
\newblock in preparation.

\bibitem{BD79}
E.~Bannai and R.~M. Damerell.
\newblock Tight spherical designs. {I}.
\newblock {\em J. Math. Soc. Japan}, 31(1):199--207, 1979.

\bibitem{BD80}
E.~Bannai and R.~M. Damerell.
\newblock Tight spherical designs. {II}.
\newblock {\em J. London Math. Soc. (2)}, 21(1):13--30, 1980.

\bibitem{BMV05}
E.~Bannai, A.~Munemasa, and B.~Venkov.
\newblock The nonexistence of certain tight spherical designs.
\newblock {\em Algebra i Analiz}, 16(4):1--23, 2004.

\bibitem{BOT15}
E.~Bannai, T.~Okuda, and M.~Tagami.
\newblock Spherical designs of harmonic index {$t$}.
\newblock {\em J. Approx. Theory}, 195:1--18, 2015.

\bibitem{BS81}
E.~Bannai and N.~J.~A. Sloane.
\newblock Uniqueness of certain spherical codes.
\newblock {\em Canad. J. Math.}, 33(2):437--449, 1981.

\bibitem{BGOY15}
A.~Barg, A.~Glazyrin, K.~A. Okoudjou, and W.~Yu.
\newblock Finite two-distance tight frames.
\newblock {\em Linear Algebra Appl.}, 475:163--175, 2015.

\bibitem{Bou02}
N.~Bourbaki.
\newblock {\em Lie groups and {L}ie algebras. {C}hapters 4--6}.
\newblock Elements of Mathematics (Berlin). Springer-Verlag, Berlin, 2002.
\newblock Translated from the 1968 French original by Andrew Pressley.

\bibitem{CS99}
J.~H. Conway and N.~J.~A. Sloane.
\newblock {\em Sphere packings, lattices and groups}, volume 290 of {\em
  Grundlehren der Mathematischen Wissenschaften [Fundamental Principles of
  Mathematical Sciences]}.
\newblock Springer-Verlag, New York, 1988.
\newblock With contributions by E. Bannai, J. Leech, S. P. Norton, A. M.
  Odlyzko, R. A. Parker, L. Queen and B. B. Venkov.

\bibitem{DGS77}
P.~Delsarte, J.~M. Goethals, and J.~J. Seidel.
\newblock Spherical codes and designs.
\newblock {\em Geometriae Dedicata}, 6(3):363--388, 1977.

\bibitem{NV13}
G.~Nebe and B.~Venkov.
\newblock On tight spherical designs.
\newblock {\em Algebra i Analiz}, 24(3):163--171, 2012.

\bibitem{OY16}
T.~Okuda and W.~Yu.
\newblock A new relative bound for equiangular lines and nonexistence of tight
  spherical designs of harmonic index 4.
\newblock {\em European J. Combin.}, 53:96--103, 2016.

\bibitem{ZBBKY17}
Y.~Zhu, E.~Bannai, Et. Bannai, K.~Kim, and W.~Yu.
\newblock On spherical designs of some harmonic indices.
\newblock {\em Electron. J. Combin.}, 24(2):Paper 2.14, 28, 2017.

\end{thebibliography}
    \bibliographystyle{plain}

\end{document}